\documentclass[12pt, reqno]{amsart}
\usepackage{eurosym}
\usepackage{amsfonts}
\usepackage{amssymb}
\usepackage{graphicx}
\usepackage{pstricks}
\usepackage{amsmath}
\usepackage{amsxtra}
\usepackage{color}
\usepackage{cite}
\usepackage{bm}
\usepackage{enumitem}

\makeatletter
\newcommand*\bigcdot{\mathpalette\bigcdot@{.5}}
\newcommand*\bigcdot@[2]{\mathbin{\vcenter{\hbox{\scalebox{#2}{$\m@th#1\bullet$}}}}}
\makeatother

\def\ker{\operatorname{ker}}

\def\sup{\operatorname{sup}}

\def\Ran{\operatorname{Ran}}
\def\shift{\operatorname{shift}}

\def\diag{\operatorname{diag}}
\def\index{\operatorname{index}}

\newcommand{\bT}{\bm{T}}
\newcommand{\dbT}{\Delta_t(\bm{T})}

\newcommand{\KT}{K(\bm{T},\mathcal{H})}
\newcommand{\KdT}{K(\Delta_t(\bm{T}),\mathcal{H})}
\newcommand{\QH}{\mathcal{Q}(\mathcal{H})}

\theoremstyle{plain}
\newtheorem{theorem}{Theorem}[section]
\newtheorem{lemma}[theorem]{Lemma}

\newtheorem{corollary}[theorem]{Corollary}

\newtheorem{definition}[theorem]{Definition}

\theoremstyle{definition}

\newtheorem{example}[theorem]{Example}
\newtheorem{remark}[theorem]{Remark}

\numberwithin{equation}{section} 
\newenvironment{changemargin}[2]{%
  \begin{list}{}{%
    \setlength{\topsep}{0pt}%
    \setlength{\leftmargin}{#1}%
    \setlength{\rightmargin}{#2}%
    \setlength{\listparindent}{\parindent}%
    \setlength{\itemindent}{\parindent}%
    \setlength{\parsep}{\parskip}%
  }%
  \item[]}{\end{list}} 

\begin{document}

\title{ The Spectral Picture and Joint Spectral Radius \\ of the Generalized Spherical Aluthge
Transform}
\author{Chafiq Benhida}
\address{UFR de Math\'{e}matiques, Universit\'{e} des Sciences et
Technologies de Lille, F-59655 \newline Villeneuve-d'Ascq Cedex, France}
\email{chafiq.benhida@univ-lille.fr}
\author{Ra\'{u}l E. Curto}
\address{Department of Mathematics, The University of Iowa, Iowa City, Iowa
52242}
\email{raul-curto@uiowa.edu}
\author{Sang Hoon Lee}
\address{Department of Mathematics, Chungnam National University, Daejeon,
34134, Republic of Korea}
\email{slee@cnu.ac.kr}
\author{Jasang Yoon}
\address{School of Mathematical and Statistical Sciences, The University of
Texas Rio Grande Valley, Edinburg, Texas 78539, USA}
\email{jasang.yoon@utrgv.edu}
\thanks{The first named author was partially supported by Labex CEMPI
(ANR-11-LABX-0007-01).}
\thanks{The second named author was partially supported by Labex CEMPI (ANR-11-LABX-0007-01).}
\thanks{The third named author was partially supported by NRF (Korea) grant
No. 2020R1A2C1A0100584611.}
\thanks{The fourth named author was partially supported by a grant from the
University of Texas System and the Consejo Nacional de Ciencia y Tecnolog%
\'{\i}a de M\'{e}xico (CONACYT)}

\begin{abstract}
For an arbitrary commuting $d$--tuple $\bT$ of Hilbert space operators, we fully determine the spectral picture of the generalized spherical Aluthge transform $\dbT$ and we prove that the spectral radius of $\bT$ can be calculated from the norms of the iterates of $\dbT$. \ Let $\bm{T} \equiv (T_1,\cdots,T_d)$ be a commuting $d$--tuple of bounded operators acting on an infinite dimensional separable Hilbert space, let $P:=\sqrt{T_1^*T_1+\cdots+T_d^*T_d}$, and let
$$
\left(
\begin{array}{c}
T_1 \\
\vdots \\
T_d
\end{array}
\right)
=
\left(
\begin{array}{c}
V_1 \\
\vdots \\
V_d
\end{array}
\right)
P
$$
be the canonical polar decomposition, with $(V_1,\cdots,V_d)$ a (joint) partial isometry and $$
\bigcap_{i=1}^d \ker T_i=\bigcap_{i=1}^d \ker V_i=\ker P.
$$
\medskip
For $0 \le t \le 1$, we define the generalized 
spherical Aluthge transform of $\bm{T}$ by 
$$
\Delta_t(\bm{T}):=(P^t V_1P^{1-t}, \cdots, P^t V_dP^{1-t}).
$$
We also let $\left\|\bm{T}\right\|_2:=\left\|P\right\|$. \ 
We first determine the spectral picture of $\Delta_t(\bm{T})$ in terms of the spectral picture of $\bm{T}$; in particular, we prove that, for any $0 \le t \le 1$,  $\Delta_t(\bm{T})$ and $\bm{T}$ have the same Taylor spectrum, the same Taylor essential spectrum, the same Fredholm index, and the same Harte spectrum. \ We then study the joint spectral radius $r(\bm{T})$, and prove that $r(\bm{T})=\lim_n\left\|\Delta_t^{(n)}(\bm{T})\right\|_2 \,\, (0 < t < 1)$, where $\Delta_t^{(n)}$ denotes the $n$--th iterate of $\Delta_t$. \ For $d=t=1$, we give an example where the above formula fails.
\end{abstract}

\maketitle
\tableofcontents

\setcounter{tocdepth}{2}


\section{Introduction} \label{Intro}

Let $\mathcal{H}$ be a complex infinite dimensional Hilbert space, let $%
\mathcal{B}(\mathcal{H})$ denote the algebra of bounded linear operators on $%
\mathcal{H}$, and let $T\in \mathcal{B}(\mathcal{H})$. \ For $T\equiv V|T|$
the canonical polar decomposition of $T$, we let $\Delta \left( T\right) \equiv \widetilde{T}:=|T|^{1/2}V|T|^{1/2}$ denote the Aluthge transform of $T$ \cite{Alu}. \ It is well known that $T$ is invertible if and only if $\widetilde{T}$ is
invertible; moreover, the spectra of $T$ and $\widetilde{T}$ are equal. \ Over the last
two decades, considerable attention has been given to the study of the
Aluthge transform; cf. \cite{ APS, CJL, DySc, FJKP, IYY, JKP1, JKP2, LLY2, Ya}. \ Moreover, the Aluthge transform has been generalized to the case of
powers of $|T|$ different from $\frac{1}{2}$ \cite{Ben,Cha, DKY, Tam} and to the case of commuting $d$--tuples of operators \cite{BenCu, BCLY, CuYo1, CuYo2, CuYo3, FeYa, KiYo1, KiYo2}. \ 

Recall that, for $t \in [0,1]$, the \textit{generalized Aluthge transform} of $T$ is $\widetilde{T^{t}}:=|T|^{t}V|T|^{1-t}$; the special instance of $t=1$ is known as the \textit{Duggal transform}, defined as $\widetilde{T}^{D}:=|T|V$ (see \cite{FJKP,KiYo2}). \ 

In this paper we will focus on the generalized {\it spherical} Aluthge
transform of commuting $d$--tuples, which we now describe. \ Let $\bm{T}\equiv (T_{1},\cdots,T_{d})$ be a commuting $d$--tuple of operators on $\mathcal{H}$, and consider the canonical polar decomposition of the
column operator 
\begin{equation} \label{eqpolar}
D_{\bm{T}}:=\left(
\begin{array}{c}
T_{1} \\
\vdots \\
T_{d}%
\end{array}%
\right);
\end{equation}
that is,
\begin{equation} 
\begin{tabular}{l}
$D_{\bm{T}}=D_{\bm{V}}P:\mathcal{H\longrightarrow H}\oplus \cdots \oplus \mathcal{H}\text{,}$%
\end{tabular}
\label{operator}
\end{equation}%
where $P:=|D_{\bT}| \equiv \sqrt{T_1^*T_1+\cdots+T_d^*T_d}$ is a positive operator on $\mathcal{H}$ and $D_{\bm{V}}:=\left(
\begin{array}{c}
V_{1} \\
\vdots \\
V_{d}
\end{array}%
\right) $ 
is a (joint) partial isometry from $\mathcal{H}$ to $\mathcal{H}\oplus \cdots \oplus \mathcal{H}$. \ Then, $(D_{\bm{V}})^*D_{\bm{V}}$ is the (orthogonal) projection onto the initial
space of the partial isometry $D_{\bm{V}}$, which in turn is
\begin{equation*}
\left(\ker D_{\bT} \right) ^{\bot }=\left( \ker T_{1}\cap \cdots \cap \ker T_{d}\right) ^{\bot }=\left(\ker P\right) ^{\bot }.
\end{equation*}%
We define the \textit{spherical} polar decomposition of $\bm{T}$ by%
\begin{equation}
\bm{T}\equiv (V_{1}P,\cdots,V_{d}P)  \label{eq2}
\end{equation}
(cf. \cite{CLY10, CuYo1, KiYo1, KiYo2}). \ We will also define the $2$--norm of $\bm{T}$ by $\left\|\bm{T}\right\|_2:=\left\|P\right\|$. \ 

For $0\leq t\leq 1$, the \textit{generalized spherical Aluthge transform} of $\bm{T}$ is the $d$--tuple 
\begin{equation}
\Delta _{t}\left( \bm{T}\right) :=(P^{t}V_{1}P^{1-t},\cdots,P^{t}V_{d}P^{1-t}).\medskip  \label{eq}
\end{equation}%
We remark that when $t=\frac{1}{2}$ (resp. $t=1$), we get the spherical
Aluthge (resp. Duggal) transform $\Delta_{\textrm{sph}}(\bm{T})$ (resp. $\widehat{%
\bm{T}}^{D}$) of $\bm{T}$; that is,%
\begin{equation}
\begin{tabular}{l}
$\Delta _{\frac{1}{2}}(\bm{T}) \equiv \Delta _{\mathrm{sph}}\left( \bm{T}\right) :=(P^{\frac{1}{2}}V_{1}P^{\frac{1}{2}},\cdots, P^{\frac{1}{2}}V_{d}P^{\frac{1}{2}})\medskip $ \\
$\text{(resp. }\Delta _{1}\left( \bm{T}\right) \equiv \widehat{\bm{T}}^{D}:=(PV_{1},\cdots,PV_{d})$.
\end{tabular}
\label{Eq}
\end{equation}

In \cite{CuYo1,CuYo2,CuYo3}, it was shown that the spherical Aluthge transform ($t=\frac{1}{2}$) respects the commutativity of the pair; the key step was the identity $V_iPV_j=V_jPV_i$, for all $i,j=1,\cdots,d$. \ Using this, one can prove that $\Delta_t(\bm{T})$ is commutative whenever $\bm{T}$ is, for all $0 \le t \le 1$ (cf. (\ref{v1pv2}) below), a fact also observed in \cite{DKY}. \ (However, the $d$--tuple $\bm{V}$ is, in general, not commutative.)  

Before we state our main results, we pause to recall the notation we will use for several spectral systems. \ We shall let $\sigma_T$ denote the {\it Taylor} spectrum, $\sigma_{Te}$ the {\it Taylor essential} spectrum, $\sigma_p$ the {\it point} spectrum, $\sigma_{\ell}$ the {\it left} spectrum (also known as the {\it approximate point} spectrum), $\sigma_r$ the {\it right} spectrum, and $\sigma_H$ the {\it Harte} spectrum ($\sigma_H:=\sigma_{\ell}\cup\sigma_r)$. \ 

\begin{theorem} \label{Taylorsp}

Let $\bm{T}=\bm{(}T_{1},\cdots,T_{d})$ be a commuting $d$--tuple of operators on $\mathcal{H}$, and let $0 < t \le 1$. Then
\begin{equation*}
\sigma _{T}(\Delta _{t}\left( \bm{T}\right))=\sigma _{T}(\bm{T}).\medskip
\end{equation*}
As a consequence, we have 
$$
r(\Delta _{t}\left( \bm{T}\right))=r_{T}\left( \bm{T}\right),
$$
where $r_T$ denotes the {\it joint spectral radius}; i.e., 
\begin{equation} \label{spra}
r(\bm{T}):=\sup \{\left\|\bm{\lambda}\right\|_2:= \sqrt{\left|\lambda_1\right|^2+\cdots+\left|\lambda_d\right|^2}:\bm{\lambda}\equiv(\lambda_1,\cdots,\lambda_d) \in \sigma_T(\bm{T}) \},
\end{equation}
and similarly for $r_{T}(\Delta _{t}\left( \bm{T}\right))$.
\end{theorem}

\begin{theorem} \label{Tayloressentialsp}
Let $\bm{T}=\bm{(}T_{1},\cdots,T_{d})$ be a commuting $d$--tuple of operators on $\mathcal{H}$, and let $0 < t \leq 1$. \ Then 
\begin{equation*}
\sigma _{Te}(\Delta _{t}\left( \bm{T}\right) )=\sigma _{Te}(\bm{T}).\medskip
\end{equation*}%
Moreover, the Fredholm index satisfies
$$
\mathrm{index }(\Delta _{t}(\bm{T})-\bm{\lambda})=\mathrm{index }(\bm{T}-\bm{\lambda}),
$$
for all $\bm{\lambda} \notin \sigma _{Te}(\bm{T})$.\medskip
\end{theorem}

\begin{theorem} \label{othersp}
Let $\bm{T}=\bm{(}T_{1},\cdots,T_{d})$ be a commuting $d$--tuple of operators on $%
\mathcal{H}$. \ Then, for $0\leq t\leq 1$, the following statements hold.
\begin{itemize}
\item[(i)] \ $\sigma _{p}(\Delta _{t}\left( \bm{T}\right) )=\sigma _{p}(\bm{T})$.
\item[(ii)] \ $\sigma _{\ell}(\Delta _{t}\left( \bm{T}\right) )=\sigma _{\ell}(\bm{T})$.
\item[(iii)] \ $\sigma _{H}(\Delta _{t}\left( \bm{T}\right) )=\sigma _{H}(\bm{T})$.
\item[(iv)] \ Assume $t>0$. \ Then $\sigma _{r}(\bm{T}) \subseteq \sigma _{r}(\Delta _{t}( \bm{T})) \subseteq \sigma_r(\bm{T}) \cup \{0\}$.
\item[(v)] \ More generally, if 
$\sigma_{\pi,k}$ and $\sigma_{\delta,k}$ denote the S\l odkowski's spectral systems \cite{Slo}, then we have \ $\sigma _{\pi,k}(\Delta _{t}\left( \bm{T}\right) )=\sigma _{\pi,k}(\bm{T}) \; (k=0,\dots,d)$ and 
$\sigma_{\delta,k}(\bm{T}) \subseteq \sigma _{\delta,k}(\Delta _{t}\left( \bm{T}\right) )\subseteq \sigma _{\delta,k}(\bm{T}) \cup \{0\} \; (k=0,\cdots,d)$. 

\end{itemize}
\end{theorem}

The first inclusion in Theorem \ref{Taylorsp}(iv) can be proper, as the following example shows.

\begin{example} \label{ex14}
For $0<t\leq 1$, consider the commuting pair $\bm{T}:=(U_+^*,0)$, where $U_+^*$ denotes the adjoint of the (unweighted) unilateral weighted shift acting on $\ell^2(\mathbb{Z}_+)$, with canonical orthonormal basis $\{e_0,e_1,\cdots\}$. \ The polar decomposition of $U_+^*$ is $U_+^*(I-E_0)$, where $E_0$ denotes the orthogonal projection onto the $1$-dimensional subspace generated by $e_0$. \ It is clear that the pair $\bm{T}$ has positive part $P=(I-E_0)$; moreover, $U_+^*(I-E_0)=U_+^*$. \ Therefore, $\Delta_t(\bm{T})=((I-E_0)^tU_+^*(I-E_0)^{1-t},0)=((I-E_0)U_+^*,0)=(U_+^*-e_0 \otimes e_1,0)$, for all $t \in \left(0,1\right]$. \ This pair is unitarily equivalent to $(0 \oplus U_+^*,0)$, and it follows that $\sigma_r(\Delta_t(\bm{T}))=(\{0\} \cup \mathbb{T}) \times \{0\}=\mathbb{T} \times \{0\} \cup \{(0,0)\}$, for all $t \in \left(0,1\right]$. \ (Here $\mathbb{T}$ denotes the closed unit circle.) \ On the other hand, $\sigma_r(\bm{T})=\mathbb{T} \times \{0\}$. \qed
\end{example}

As we will see in Section \ref{SpecPic}, the key ingredient in the Proof of Theorem \ref{Taylorsp} is a careful analysis of the role of $P$ in establishing an isomorphism between the Koszul complexes of $\bm{T}$ and $\Delta_t(\bm{T})$. \ As a consequence, we shall see in Section \ref{SpecPic} that $P$ also enters in the Proof of Theorem \ref{othersp}.

We now turn our attention to the spectral radius of $\bm{T}$. \ For $n=1$, T. Yamazaki proved in \cite{Ya} that the Aluthge transform of an operator $T$ can be used to calculate the spectral radius of $T$, via iteration. \ Concretely, $r(T)=\lim_n \left\|\Delta^{(n)}(T)\right\|$, where $\Delta^{(n)}$ denotes the $n$--th iterate of $\Delta$.

We first extend this result to the case of single operators and the generalized Aluthge transform.

\begin{theorem} \label{thm15}
For $0 \le t \le 1$, there exists $s_t\ge 0$ such that $r\left(
T\right) \leq s_t \leq \left\Vert T\right\Vert $ and
\begin{equation*}
\lim_{n\rightarrow \infty } \left\Vert \Delta _{t}^{\left( n\right)
}\left( T\right) \right\Vert =s_t.\medskip
\end{equation*}
\end{theorem}

\begin{theorem} \label{thm16}
For $0 \le t < 1$, we have
\begin{equation} \label{eq16}
\lim_{n\rightarrow \infty }\left\Vert \Delta _{t}^{\left( n\right)
}\left( T\right) \right\Vert =r\left( T\right) .\medskip
\end{equation}
\end{theorem}

We now extend the above results to the case of commuting $d$--tuples. \ For $t=\frac{1}{2}$, Theorem \ref{thm18} was first proved in \cite{FeYa}.

\begin{theorem} \label{thm17}
Let $\bm{T}=\bm{(}T_{1},\cdots,T_{d})$ be a commuting $d$--tuple of operators on $\mathcal{H}$. \ For $t\in \left[ 0,1\right] $ and $n \ge 1$, the sequence $%
\left\{ \left\Vert \Delta_{t}^{\left( n\right)
}\left( \bm{T}\right) \right\Vert _{2}\right\}_{n=1}^{\infty}$ satisfies 
\begin{equation*} 
r_{T}\left( \bm{T}\right) \le \left\Vert \Delta _{t}^{\left( n+1 \right) }\left( \bm{T}\right) \right\Vert _{2} \leq \left\Vert \Delta _{t}^{\left( n\right) }\left( \bm{T}\right) \right\Vert _{2}\leq
\left\Vert \bm{T}\right\Vert _{2}.\medskip
\end{equation*}
\end{theorem}

\begin{theorem} \label{thm18}
Let $\bm{T}=\bm{(}T_{1},\cdots,T_{d})$ be a commuting $d$--tuple of operators on $\mathcal{H}$. \ For $0 < t < 1$, we have%
\begin{equation} \label{eq17}
r_{T}\left( \bm{T}\right) = \lim_{n\rightarrow \infty } \left\Vert \Delta _{t}^{\left( n\right) }\left( \bm{T}\right) \right\Vert
_{2} .\medskip
\end{equation}
\end{theorem}

The organization of the paper is as follows. \ In Section \ref{Notation} we collect notation and terminology needed throughout the paper, and we also list some standard results needed for our proofs. \ We devote Section \ref{SpecPic} to the spectral results, which allow us to compare the spectral pictures of $\bm{T}$ and $\Delta_t(\bm{T})$, for all $t \in [0,1]$. \ In Section \ref{JointSpecRad} we discuss the multivariable version of the spectral radius formula for the generalized spherical Aluthge transform. \ We extend the recent result of K. Feki and T. Yamazaki \cite{FeYa} for $t=\frac{1}{2}$ to the general case of $0 < t < 1$; we also discuss what happens when $t=1$. \ The proof of Theorem \ref{thm18} is rather subtle, in that we need to study in great detail the behavior of the norms of $\bT$, its generalized spherical Aluthge transform $\dbT$, the successive iterates $\Delta_T^{(n)}(\bT)$, and the norms of the respective powers, of the form $\left\|(\Delta_t^{(n)}(\bT))^k\right\|_2$. \ There is a rather fascinating interplay between the norms of the powers and the norms of the iterates.  

\section{Notation and Preliminaries} \label{Notation}

Let $\bm{T} \equiv (T_1,T_2)$ be a commuting pair of operators on $\mathcal{H}$, and consider the Koszul complex of $\bm{T}$ on $\mathcal{H}$; that is,
\begin{equation*}
K(\bm{T},\mathcal{H}):\; \; 0 \longrightarrow \mathcal{H} 
\xrightarrow{
\left(
\begin{array}{c}
T_{1} \\
T_{2}
\end{array}
\right)}
\mathcal{H} \oplus \mathcal{H}
\xrightarrow{\left(
\begin{array}{cc}
-T_{2} & T_{1}
\end{array}
\right) }\mathcal{H} \longrightarrow 0.
\end{equation*}
\medskip
Recall that $P=\sqrt{T_1^*T_1+T_2^*T_2}$ and that the generalized Aluthge transform is given by $\Delta_t(\bm{T})=\left(P^tV_1P^{1-t},P^tV_2P^{1-t}\right)$, for $0 \le t \le 1$, where $(V_1,V_2)$ is the joint partial isometry in the polar decomposition of $\left(
\begin{array}{c}
T_{1} \\
T_{2}%
\end{array}%
\right)$. \ We know that 
\begin{equation} \label{v1pv2}
(V_1PV_2-V_2PV_1)P=T_1T_2-T_2T_1=0,    
\end{equation}
and therefore the operator $C:=V_1PV_2-V_2PV_1$ vanishes on the range of $P$. \ Since $C$ also vanishes on the kernel of $P$ (because $\ker P=\ker V_1 \cap \ker V_2$), we easily conclude that $C=0$. \ A direct consequence of this is the commutativity of $\Delta_t(\bm{T})$ for all $0 \le t \le 1$. \ In an entirely similar way one can prove that for a commuting $d$--tuple $\bm{T} \equiv (T_1,\cdots,T_d)$ the identities 
\begin{equation} \label{v1pv2A}
V_iPV_j=V_jPV_i
\end{equation}
always hold ($1 \le i,j \le d$), and therefore the generalized spherical Aluthge transform $\Delta_t(\bm{T})$ is commutative for all $t \in [0,1]$. \ It follows that we can form two Koszul complexes, one for $\bm{T}$ and one for $\Delta_t(\bm{T})$, with boundary maps denoted by $D_{\bm{T}}^0,D_{\bm{T}}^1,\cdots,D_{\bm{T}}^{d-1}$ and by $D_{\Delta_t(\bm{T})}^0,D_{\Delta_t(\bm{T})}^1,\cdots,D_{\Delta_t(\bm{T})}^{d-1}$, respectively:
\begin{equation*}
K(\bm{T},\mathcal{H}):\; \; 0 \longrightarrow \mathcal{H} 
\xrightarrow{D_{\bm{T}}^0}\mathcal{H} \oplus \cdots \oplus \mathcal{H}
\xrightarrow{D_{\bm{T}}^1} \cdots \xrightarrow{D_{\bm{T}}^{d-2}}\mathcal{H} \oplus \cdots \oplus \mathcal{H} \xrightarrow{D_{\bm{T}}^{d-1}} \mathcal{H} \longrightarrow 0
\end{equation*}
and
\begin{equation*}
K(\Delta_t(\bm{T}),\mathcal{H}):\; \; 0 \longrightarrow \mathcal{H} 
\xrightarrow{D_{\Delta_t(\bm{T})}^0}\mathcal{H} \oplus \cdots \oplus \mathcal{H}
\xrightarrow{D_{\Delta_t(\bm{T})}^1} \cdots \xrightarrow{D_{\Delta_t(\bm{T})}^{d-2}}\mathcal{H} \oplus \cdots \oplus \mathcal{H} \xrightarrow{D_{\Delta_t(\bm{T})}^{d-1}} \mathcal{H} \longrightarrow 0 .
\end{equation*}
(For the precise definition of the Koszul complex and its boundary maps, the reader is referred to \cite{Tay1,Appl,Vas}; observe that $D_{\bm{T}}^0=D_{\bm{T}}$, as defined in (\ref{eqpolar}).)

From the definition of $\Delta_t(\bm{T})$ we readily obtain $d$ key identities:
\begin{equation*}
\begin{array}{lrcl}
(1)&(P^t \oplus \cdots \oplus P^t)D_{\bm{T}}^0 & = & D_{\Delta_t(\bm{T})}^0 P^t; \\
(2)&(P^t \oplus \cdots \oplus P^t)D_{\bm{T}}^1&=&D_{\Delta_t(\bm{T})}^1(P^t \oplus \cdots \oplus P^t); \\
\; \; \vdots&\vdots&\vdots&\vdots \\
(d)&P^t D_{\bm{T}}^{d-1}&=&D_{\Delta_t(\bm{T})}^{d-1}(P^t \oplus \cdots \oplus P^t).
\end{array}
\end{equation*}
\medskip
These identities establish, for each $0 \le t \le 1$, a co-chain homomorphism $\Phi_t: K(\bm{T},\mathcal{H}) \longrightarrow K(\Delta_t(\bm{T}),\mathcal{H})$. \label{page5} \ This homomorphism will allow us to compare, at each stage, the exactness of one complex with the exactness of the other. \ We will do this in Section \ref{SpecPic}.

\begin{remark}
Observe that, for $0<t\le1$, the fixed points of the generalized spherical Aluthge transform are the {\it spherically quasinormal} $d$--tuples; these are the $d$--tuples for which $P$ commutes with each $T_i$, or equivalently, with each $V_i$, for $i=1,\cdots,d$. \ For, if $0<t \le 1$ and $\Delta_t(\bT)=\bT$, we must have $(P^tV_i-V_iP^t)P^{1-t}=0$ for $i=1,\cdots,d$. \ Let $C_i:=P^tV_i-V_iP^t \; (i=1,\cdots,d)$. \ When $t<1$, $C_i$ vanishes on the range of $P^{1-t}$, and therefore it vanishes on the range of $P$ (for all $i=1,\cdots,d$). \ Since it also vanishes on the kernel of $P$, we conclude that $C_i=0 \; (i=1,\cdots,d)$. \ It follows that $V_i$ commutes with $P$, for all $i=1,\cdots,d$. \ When $t=1$, the commutativity of $V_i$ and $P$ ($i=1,\cdots,d$) is obvious.  \qed
\end{remark}

We now recall the definitions of Taylor and Taylor essential spectrum, and of Fredholm index.

\begin{definition} \label{def21}
A commuting $d$--tuple $\bm{T} \equiv (T_1,\cdots,T_d)$ is said to be \textit{(Taylor) invertible} if
its associated Koszul complex $K(\bm{T,}\mathcal{H})$ is exact (at each stage). \ The
Taylor spectrum of $\bm{T}$ is
\begin{equation*}
\begin{tabular}{l}
$\sigma _{T}(\bm{T}):=\{\bm{\lambda} \equiv (\lambda _{1},\cdots,\lambda _{d})\in \mathbb{C}
^{d}:K(\bm{T}-\bm{\lambda},\mathcal{H}) \text{ is not exact}\} $.
\end{tabular}%
\end{equation*}
$\bm{T}$ is said to be \textit{Fredholm} if all the homology quotients in $K(\bm{T},\mathcal{H})$ are finite-dimensional; this implies, in particular, that the boundary maps have closed range. \ The Taylor essential spectrum is the set
\begin{equation*}
\begin{tabular}{l}
$\sigma _{Te}(\bm{T}):=\{\bm{\lambda} \in \mathbb{C}
^{d}:\bm{T}-\bm{\lambda} \text{ is not Fredholm}\}$.
\end{tabular}
\end{equation*}
\end{definition}

If, given $k=0,1, \cdots,d$, one only requires that the Koszul complex be exact at the last $k+1$ stages $d-k,d-k+1,\cdots,d-1,d$, we obtain Z. S\l odkowski's spectral systems $\sigma_{\delta,k}$. \ These spectral systems lie between the right spectrum and the Taylor spectrum; that is,
$\sigma_r \equiv \sigma_{\delta,0} \subseteq \sigma_{\delta,1} \subseteq \cdots \subseteq \sigma_{\delta,d-1} \subseteq \sigma_{\delta,d} \equiv \sigma_T$. \ Similarly, if, given $k=0,1,\cdots,d$, one only requires that the Koszul complex be exact at the first $k+1$ stages $0,1,\cdots,k-1,k$, together with the closed range of $D_{\bm{T}}^k$, we obtain the spectral system $\sigma_{\pi,k}$. \ These spectral systems lie between the left spectrum and the Taylor spectrum; that is,
$\sigma_{\ell} \equiv \sigma_{\pi,0} \subseteq \sigma_{\pi,1} \subseteq \cdots \subseteq \sigma_{\pi,d} \equiv \sigma_T$. \ If $\sigma_{Z}$ denotes one of the above mentioned spectral systems, we will say that $\bm{T}$ is Z--invertible whenever $\bm{0} \equiv (0,\cdots,0) \notin \sigma_{Z}(\bm{T})$; for instance, we will refer to $\bm{T}$ as being Taylor invertible, or left invertible, or Harte invertible, and so on.  

\newpage
J.L. Taylor showed in \cite{Tay1} that, if $\mathcal{H}\neq \{0\}$, then $%
\sigma _{T}(\bm{T})$ is a nonempty, compact subset of the polydisc of
multiradius $(r(T_{1}),\cdots,r(T_{d})),$ where $r(T_{i})$ is the
spectral radius of $T_{i}$ \ ($i=1,\cdots,d$). \ As a consequence, if $\bm{\lambda} \in \sigma_T(\bm{T})$, then
$\left\|\bm{\lambda}\right\|_2 \le \sqrt{d}\left\|P\right\|$; for, $|\lambda_i|^2 \le r(T_i)^2 \le \left\|T_i^*T_i\right\|$ for every $i=1,\cdots,d$, and therefore $\left\|\bm{\lambda}\right\|_2^2 \le \sum_{i=1}^d \left\|T_i^*T_i\right\| \le d\left\|P\right\|^2$. \ However, one can do better. \ For, consider the elementary operator $M_{\bT}:\mathcal{B}(\mathcal{H})\rightarrow \mathcal{B}(\mathcal{H})$ given as $M_{\bT}(X):=\sum_{i=1}^dT_i^*XT_i \; (X \in \mathcal{B}(\mathcal{H}))$. \ We know from \cite{Cu3} that $\sigma(M_{\bT})=\{\bar{\lambda_1}\lambda_1+\cdots+\bar{\lambda_d}\lambda_d: \bm{\lambda} \in \sigma_T(\bT) \}$. \ As a result, if $\bm{\lambda} \in \sigma_T(\bT)$, we have $\left\|\bm{\lambda}\right\|_2^2 \le r(M_{\bT}) \le \left\|M_{\bT}\right\|=\left\|P\right\|^2$. \ Thus, $r(\bT) \le \left\|P\right\|=\left\|\bT\right\|_2$. \ (For additional facts about these joint spectra, the reader is referred to \cite{Cu1,Cu2,Cu3,Appl,Tay2, Slo, SZ}.)\medskip

As shown in \cite{Cu2,Appl}, the Fredholmness of $\bm{T}$ can
be detected in the Calkin algebra $\mathcal{Q}(\mathcal{H}):=\mathcal{B}(%
\mathcal{H})/\mathcal{K}(\mathcal{H})$. \ (Here $\mathcal{K}$ denotes the
closed two--sided ideal of compact operators on $\mathcal{H}$; we also let $\pi :\mathcal{B}(%
\mathcal{H})\longrightarrow \mathcal{Q}(\mathcal{H})$ denote the quotient
map.) \ Concretely, $\bm{T}$ is Fredholm on $\mathcal{H}$ if and only if
the $d$--tuple of left multiplication operators $L_{\pi (\bm{T})}:=(L_{\pi
(T_{1})},\cdots,L_{\pi (T_{d})})$ is Taylor invertible when acting on $\mathcal{Q}(\mathcal{H})$. \ In particular, $\bm{T}$ is left Fredholm on $\mathcal{H}
$ if and only if $L_{\pi (\bm{T})}$ is left invertible on $\mathcal{Q}(\mathcal{H})$.

Given a commuting $d$--tuple $\bm{T}$, by the {\it spectral picture} of $\bm{T}$, denoted $SP(\bm{T})$, we will refer to the collection of the sets $\sigma_T(\bm{T})$, $\sigma_{Te}(\bm{T})$, $\sigma_{\pi,k}(\bm{T})$, $\sigma_{\delta,k}(\bm{T})$ and $\sigma_H(\bm{T})$ (for all $k=0,1,\cdots,d$), together with the index function defined on the Fredholm domain of $\bm{T}$ (that is, the complement in $\mathbb{C}^d$ of $\sigma_{Te}(\bm{T})$).

Our strategy to describe the spectral picture of $\Delta_t(\bm{T})$ will partially rely on the theory of criss--cross commutativity for pairs of $d$--tuples of operators acting on $\mathcal{H}$. \ Given two $d$--tuples $\bm{A}$ and $\bm{B}$, we say that $\bm{A}$ and $\bm{B}$ {\it criss--cross commute} if $A_iB_jA_k=A_kB_jA_i$ and $B_iA_jB_k=B_kA_jB_i$, for all $i,j,k=1,\cdots,d.$ \ (Observe that in this definition we do not assume that $\bm{A}$ or $\bm{B}$ is commuting.) \label{crisscross} \ We now let 
$$
\bm{AB}:=(A_1B_1,\cdots,A_dB_d) \; \textrm{ and } \; \bm{BA}:=(B_1A_1,\cdots,B_dA_d).
$$
It is an easy exercise to prove that if $\bm{A}$ and $\bm{B}$ criss--cross commute, then each of $\bm{AB}$ and $\bm{BA}$ is commuting. 

Now, fix a commuting $d$--tuple $\bm{T}$ and a real number $t \in [0,1]$, and let $\bm{A}\equiv(A_1,\cdots,A_d):=(V_1P^{1-t},\cdots,V_dP^{1-t})$ and $\bm{B} \equiv (B_1,\cdots,B_d):=(P^t,\cdots,P^t)$. \ Then $\bm{AB}=\bm{T}$ and $\bm{BA}=\Delta_t(\bm{T})$\label{ABBA}; moreover, the commutativity of $\bm{T}$ implies that $\bm{A}$ and $\bm{B}$ criss--cross commute (using (\ref{v1pv2A})). \ It follows that $\Delta_t(\bm{T})$ is also commuting.

We now briefly describe unilateral weighted shifts, since we use them to generate examples. \ For $\omega \equiv \{\omega
_{n}\}_{n=0}^{\infty }$ a bounded sequence of positive real numbers (called {\it weights}), let $W_{\omega }\equiv \shift (\omega _{0},\omega
_{1},\cdots ):\ell ^{2}(\mathbb{Z}_{+})\rightarrow \ell ^{2}(\mathbb{Z}_{+})$
be the associated unilateral weighted shift, defined by 
$$W_{\omega
}e_{n}:=\omega _{n}e_{n+1}\;(\textrm{all } n\geq 0),
$$
where $\{e_{n}\}_{n=0}^{\infty
}$ is the canonical orthonormal basis in $\ell ^{2}(\mathbb{Z}_{+})$. \ As we noted in Example \ref{ex14}, $U_{+}:=\shift (1,1,\cdots )$ is the (unweighted) unilateral shift, and for $0<a<1$ we will let $S_{a}:=\shift (a,1,1,\cdots )$.\medskip
 
We will also need $2$-variable weighted shifts. \ Consider the nonnegative quadrant 
$$
\mathbb{Z}_{+}^{2}:=\{\bm{k} \equiv (k_{1},k_{2}) \in \mathbb{Z}^2:k_{1} \ge 0 \text{ and }k_{2}\ge 0 \},
$$
and the Hilbert space $\ell ^{2}(\mathbb{Z}_{+}^{2})$ of square-summable complex sequences indexed by $
\mathbb{Z}_{+}^{2}$; the canonical orthonormal basis for this space is $\{e_{\bm{k}}\}_{\bm{k} \in \mathbb{Z}_+^2}$. \ Observe that $\ell ^{2}(\mathbb{Z}_{+}^{2})$ is canonically isometrically isomorphic to the Hilbert space tensor product $\ell^2(\mathbb{Z}_+) \otimes \ell^2(\mathbb{Z}_+)$, via the map $e_{\bm{k}} \mapsto e_{k_1} \otimes e_{k_2} \; (\bm{k} \equiv (k_1,k_2) \in \mathbb{Z}_+^2)$. \ 

Assume now that we are given two double-indexed positive bounded sequences $\alpha _{\bm{k}},\beta _{\bm{k}}\in \ell ^{\infty }(\mathbb{Z}_{+}^{2})$, where $\bm{k} \in \mathbb{Z}_{+}^{2}$. \ We can then define the $2$--{\it variable weighted shift} $W_{(\alpha ,\beta )}\equiv (T_{1},T_{2})$\ by
\begin{equation*}
T_{1}e_{(k_{1},k_{2})}:=\alpha _{(k_{1},k_{2})}e_{(k_{1}+1,k_{2})} \; \text{ and
} \; T_{2}e_{(k_{1},k_{2})}:=\beta _{(k_{1},k_{2})}e_{(k_{1},k_{2}+1)}.
\end{equation*}
For all $\bm{k} \in \mathbb{Z}_{+}^{2}$, we have
\begin{equation}
T_{1}T_{2}=T_{2}T_{1}\Longleftrightarrow \beta _{(k_{1}+1,k_{2})}\alpha
_{(k_{1},k_{2})}=\alpha _{(k_{1},k_{2}+1)}\beta _{(k_{1},k_{2})}.
\label{commuting}
\end{equation}

Given two unilateral weighed shifts $W_{\omega}$ and $W_{\tau}$, a trivial way to build a $2$-variable weighted shift is to let $\alpha_{(k_1,k_2)}:=\omega_{k_1}$ and $\beta_{(k_1,k_2)}:=\tau_{k_2}$ for all $\bm{k} \in \mathbb{Z}_+^2$. \ It is not hard to see that, in this case, $W_{(\alpha,\beta)} \cong (W_{\omega} \otimes I, I \otimes W_{\tau})$. \ For additional facts about $2$-variable weighted shifts, the reader is referred to \cite{CuYo,CLY10,CuYo1,CuYo2}.\medskip

For a given $\bT$, the calculation of $\dbT$ is not always straightforward, even for relatively simple forms of the components of $\bT$. \ We now present an example, which helps to visualize the type of complications one may encounter.

\begin{example}
\label{Ex1} Let $\mathcal{H}:=\ell^2(\mathbb{Z}_+) \otimes \ell^2(\mathbb{Z}_+)$, let $\mathbf{T}=\mathbf{(}T_{1},T_{2})\cong (I\otimes
U_{+}^{\ast },U_{+}\otimes I)$, and fix $0<t\leq 1$. \ Then  
\[
\begin{tabular}{l}
$\sigma _{r}(\mathbf{T})=\sigma _{r}(\Delta_t(\mathbf{T})=\mathbb{T}\times \overline{\mathbb{D}}$.\medskip 
\end{tabular}%
\]
Observe first that
\[
P^{2}=I\otimes \left( E_{0}+2E_{0}^{\perp }\right),
\]%
where $E_{0}$ is the projection onto $\left\langle e_{0}\right\rangle $ and $%
E_{0}^{\perp }:=I-E_{0}$. \ From the functional calculus for $P$, we get%
\[
P=I\otimes \left( E_{0}+\sqrt{2}E_{0}^{\perp }\right), \;
P^t=I\otimes \left( E_{0}+2^{t/2}E_{0}^{\perp }\right), \textrm{ and } \;  
P^{1-t}=I\otimes \left( E_{0}+2^{(1-t)/2}E_{0}^{\perp }\right). 
\]%
Also, observe that $\ker P=0$. \ Let $V_{1}:=I\otimes \frac{1}{%
\sqrt{2}}U_{+}^{\ast }$. \ Then, 
\begin{eqnarray*}
V_{1}P &=&I\otimes \lbrack \frac{1}{\sqrt{2}}U_{+}^{\ast }\left( E_{0}+\sqrt{%
2}E_{0}^{\perp }\right) ]=I\otimes \lbrack \frac{1}{\sqrt{2}}\left(
U_{+}^{\ast }E_{0}+\sqrt{2}U_{+}^{\ast }E_{0}^{\perp }\right) ]\medskip  \\
&=&I\otimes U_{+}^{\ast }E_{0}^{\perp }=I\otimes U_{+}^{\ast } .
\end{eqnarray*}%
Let $V_{2}:=U_{+}\otimes \left( E_{0}+\frac{1}{\sqrt{2}}E_{0}^{\perp
}\right) $. \ Then, 
\[
V_{2}P=U_{+}\otimes \left( E_{0}+\frac{1}{\sqrt{2}}E_{0}^{\perp }\right)
\left( E_{0}+\sqrt{2}E_{0}^{\perp }\right) =U_{+}\otimes I . \medskip 
\]%
Thus, 

\begin{eqnarray*} 
V_1^*V_1+V_2^*V_2 &=& I \otimes(2^{-\frac{1}{2}}U_+ \cdot 2^{-\frac{1}{2}}U_+^*) \\
&& + [U_+^* \otimes (E_0+2^{-\frac{1}{2}}E_0^{\perp})][U_+ \otimes (E_0+2^{-\frac{1}{2}}E_0^{\perp})] \\
&=& I \otimes \frac{1}{2} U_+U_+^*+U_+^*U_+ \otimes (E_0 + \frac{1}{2}E_0^{\perp}) \\
&=& I \otimes \frac{1}{2} E_0^{\perp} + I \otimes (E_0+\frac{1}{2} E_0^{\perp}) \\
&=& I \otimes I.
\end{eqnarray*}

Therefore, $\left( V_{1},V_{2}\right) $ is a (noncommuting) joint isometry,
and $\ker V_{1}\cap \ker V_{2}=0=\ker P$. \ It follows that $\left( V_{1}P,V_{2}P\right) $ is the spherical polar
decomposition of $\mathbf{(}T_{1},T_{2})$. \ Let $0<t\leq 1$. \ Then,

\begin{eqnarray*}
P^{t}V_{1}P^{1-t} &=&[I\otimes \left( E_{0}+2^{\frac{t}{2}}E_{0}^{\perp
}\right) ]\left( I\otimes \lbrack 2^{-\frac{1}{2}}U_{+}^{\ast }\right)
[I\otimes \left( E_{0}+2^{\frac{1-t}{2}}E_{0}^{\perp }\right) ]\medskip  \\
&=&I\otimes \lbrack \left( E_{0}+2^{\frac{t}{2}}E_{0}^{\perp }\right) 2^{%
-\frac{t}{2}}U_{+}^{\ast }E_{0}^{\perp }]=I\otimes \left( 2^{-\frac{t}{2}%
}E_{0}U_{+}^{\ast }+E_{0}^{\perp }U_{+}^{\ast }\right) .
\end{eqnarray*}%

On the other hand,%
\begin{eqnarray*}
P^{t}V_{2}P^{1-t} &=&[I\otimes \left( E_{0}+2^{\frac{t}{2}}E_{0}^{\perp
}\right) ][U_{+}\otimes \left( E_{0}+2^{-\frac{1}{2}}E_{0}^{\perp }\right)
][I\otimes \left( E_{0}+2^{\frac{1-t}{2}}E_{0}^{\perp }\right) ]\medskip  \\
&=&U_{+}\otimes I.
\end{eqnarray*}
Thus, 
\[
\sigma _{r}(\Delta_t(\bT)) =\sigma _{r}(2^{-\frac{t}{2}}E_{0}U_{+}^*+E_{0}^{\perp }U_{+}^*) \times \overline{\mathbb{D}}.
\]
We now need to study the right spectrum of $W(t):=2^{-\frac{t}{2}}E_{0}U_{+}^*+E_{0}^{\perp }U_{+}^*$. \ A moment's thought reveals that $W(t)$ is the adjoint of $S_{a}\equiv \shift(a,1,1,\cdots )$, where $a:=2^{-\frac{t}{2}}$. \
Then, 
\[
\sigma _{r}(W(t))=\sigma _{r}(S_{a}^*) =\overline{\sigma _{\ell }(S_{a})}=\mathbb{T},
\]
where the bar over $\sigma_{\ell}$ denotes complex conjugation. \ It follows that
\begin{equation}
\sigma _{r}(\Delta _t(\bT))=\mathbb{T}\times \overline{\mathbb{D}}.  \label{equa1}
\end{equation}%
On the other hand, 
\begin{equation}
\sigma _{r}(\bT)=\sigma _{r}(I\otimes U_{+}^*,U_{+}\otimes I)=\sigma _{r}(U_{+}^*) \times \sigma _{r}(U_{+})=\mathbb{T}\times 
\overline{\mathbb{D}}.  \label{equa2}
\end{equation}%
Therefore, by (\ref{equa1}) and (\ref{equa2}) we have $\sigma _{r}(\dbT)=\sigma_r(\bT)$. \qed
\end{example}


\section{The Spectral Picture of the Generalized Spherical Aluthge Transform} \label{SpecPic}

In \cite{JKP2}, I.B. Jung, E. Ko and C. Pearcy proved that, for $T\in \mathcal{B}(%
\mathcal{H})$, $\sigma(\Delta_{\frac{1}{2}}(T))=\sigma(T)$; subsequently, M. Ch\=o, I. Jung, and W.Y. Lee also proved that $\sigma(\Delta_1(T))=\sigma(T)$ \cite{CJL}. \ It is also known that for $r>\epsilon \geq 0$ and $%
T\equiv V|T|^{{}}\in \mathcal{B}(\mathcal{H})$, $T$ and $|T|^{\epsilon
}V|T|^{r-\epsilon }$ have the same spectrum (\cite[Lemma 5]{Hur}). \ If we
put $r=1$ and $\epsilon =t$, then $T$ and $\Delta_t(T)$ have the same spectrum, for all $0\leq t \le 1$. \ As a consequence, we obtain $r\left( T\right) =r(\Delta_t(T))$ for all $0\leq t \le 1$,
where $r\left( T\right) $ denotes the spectral radius.\medskip

In what follows, we will extend these results to the case of commuting $d$--tuples of operators on $\mathcal{H}$. \ We will first look at Taylor invertibility, which requires an analysis of the exactness of the Koszul complexes $K(\bm{T}-\bm{\lambda},\mathcal{H})$ and $K(\Delta(\bm{T})-\bm{\lambda},\mathcal{H})$, for $\bm{\lambda} \in \mathbb{C}^d$. \  We will split the discussion into two cases: (i) $\bm{\lambda}=\bm{0}$, and (ii) $\bm{\lambda} \ne \bm{0}$. \ We will first consider the case $\bm{\lambda}=\bm{0}$.

\subsection{Taylor invertibility when $\bm{\lambda}=\bm{0}$} \label{Taylor}
\ Given a commuting $d$--tuple $\bm{T} \equiv (T_1,\cdots,T_d)$, we begin with a key connection between the invertibility of $P$ and the Koszul complex homomorphism $\Phi_t$ introduced on page \pageref{page5}, right before Definition \ref{def21}. \ We recall that $\bm{T}$ is left invertible if and only if $P$ is invertible.

\begin{theorem} \label{thmP}
Let $\bm{T}$ be a commuting $d$--tuple, let $K(\bm{T},\mathcal{H})$ be its associated Koszul complex, and let $P$ be the positive factor in the polar decomposition of $D_{\bm{T}}^0$, that is, $P=\sqrt{T_1^*T_1+\cdots+T_d^*T_d}$. \ Assume that $0<t<1$. \ The following statements are equivalent: \newline
(i) \ $\bm{T}$ is left invertible; \newline
(ii) \ $P$ is invertible; \newline
(iii) \ $\Delta_t(\bm{T})$ is left invertible.
\end{theorem}

\begin{proof}

$(i) \Leftrightarrow (ii)$: This is straightforward.

$(ii) \Rightarrow (iii)$: \ Assume that $P$ is invertible. \ Then the map $\Phi_t$ is an isomorphism of Koszul complexes. \ We also know that $\bm{T}$ is left invertible; that is, the Koszul complex $K(\bm{T},\mathcal{H})$ is exact at stage $0$ and the range of $D_{\bm{T}}^0$ is closed. \ Since $\Phi_t$ is an isomorphism, we must therefore have that $K(\Delta_t(\bm{T}),\mathcal{H})$ is exact at stage $0$ and the range of $D_{\Delta_t(\bm{T})}^0$ is closed. \ Thus, $\Delta_t(\bm{T})$ is left invertible.

$(iii) \Rightarrow (ii)$: \ Observe that $(D_{\Delta(\bm{T})}^0)^*D_{\Delta(\bm{T})}^0=P^{1-t}(V_1^*PV_1+\cdots+V_d^*PV_d)P^{1-t}$. \ By assumption, $D_{\Delta(\bm{T})}^0$ is left invertible, and therefore $P^{1-t}(V_1^*PV_1+\cdots+V_d^*PV_d)P^{1-t}$ is onto. \ It follows that $P^{1-t}$ is onto, which implies that $P$ is invertible.
 
\end{proof}

\begin{remark} \label{remP}
Careful study of the Proof of Theorem \ref{thmP} reveals that left invertibility can be replaced by the invertibility associated with any of the spectral systems $\sigma_{\pi,k}$, for $0 \le k \le d$. \ For, if $0 \notin \sigma_{\pi,k}(\bm{T})$, then $0 \notin \sigma_{\pi,0}(\bm{T})$, which means that $\bm{T}$ is left invertible, and therefore $P$ is invertible, and the two Koszul complexes are isomorphic. \ We leave the details to the reader. \qed
\end{remark}

We now deal with right invertibility (which will also include invertibility relative to the spectral systems $\sigma_{\delta,k}$).

\begin{theorem} \label{thmrP}
Let $\bm{T}$ be a commuting $d$--tuple, let $K(\bm{T},\mathcal{H})$ be its associated Koszul complex, and let $P$ be the positive factor in the polar decomposition of $D_{\bm{T}}^0$. \ Assume that $0<t\le 1$, and consider the following statements. \newline
(i) \ $\bm{T}$ is right invertible; \newline
(ii) \ $P$ is invertible; \newline
(iii) \ $\Delta_t(\bm{T})$ is right invertible. \newline
Then $(iii) \Rightarrow (ii)$, and $(iii) \Rightarrow (i)$. \ The implication $(i) \Rightarrow (iii)$ is not true in general (cf. Example \ref{ex14}). 
\end{theorem}

\begin{proof}
$(iii) \Rightarrow (ii)$: \ Assume that $\Delta_t(\bm{T})$ is right invertible; that is, the boundary map $D_{\Delta_t(\bm{T})}^{d-1}$ is onto. \ Now recall that $D_{\Delta_t(\bm{T})}^{d-1}(D_{\Delta_t(\bm{T})}^{d-1})^*=P^t(\sum_{i=1}^d V_iP^{2(1-t)}V_i^*)P^{t}$. \ Since $t>0$, it follows that $P^t$ is onto, and therefore $P$ is invertible.

$(iii) \Rightarrow (i)$: \ Assume that $\Delta_t(\bm{T})$ is right invertible. \ By what we just proved, $P$ is invertible. \ It follows that the two Koszul complexes are isomorphic, and therefore $\bm{T}$ is right invertible.
\end{proof}

\begin{remark} \label{remEP}
Just as in the case of Remark \ref{remP}, the Proof of Theorem \ref{thmrP} reveals that right invertibility can be replaced by the invertibility associated with any of the spectral systems $\sigma_{\delta,k}$, for $0 \le k \le d$. \ For, if $0 \notin \sigma_{\delta,k}(\Delta_t(\bm{T}))$, then $0 \notin \sigma_{\delta,0}(\Delta_t(\bm{T}))$, which means that $\Delta_t(\bm{T})$ is right invertible, and therefore $P$ is invertible, and the two Koszul complexes are isomorphic. \ Again, we leave the details to the reader. \qed
\end{remark}

\begin{corollary} \label{Cor1}
Let $\bm{T}$ be a commuting $d$--tuple, let $0 < t \le 1$, and let $\Delta_t(\bm{T})$ be the generalized spherical Aluthge transform of $\bm{T}$. \ The following statements are equivalent. \newline
(i) $\bm{T}$ is Taylor invertible. \newline
(ii) $\Delta_t(\bm{T})$ is Taylor invertible.
\end{corollary}

\begin{proof}
Assume first that $\bm{T}$ is Taylor invertible. \ Then it is left invertible, and by Theorem \ref{thmP}, the operator $P$ is invertible. \ As a result, the Koszul complexes $K(\bm{T},\mathcal{H})$ and $K(\Delta_t(\bm{T}),\mathcal{H})$ are isomorphic. \ Since $K(\bm{T},\mathcal{H})$ is exact, so is $K(\Delta_t(\bm{T}),\mathcal{H})$. \ Therefore, $\Delta_t(\bm{T})$ is Taylor invertible.

Conversely, assume that $\Delta_t(\bm{T})$ is Taylor invertible, and therefore right invertible. \ By Theorem \ref{thmrP}, $P$ is invertible, and it follows that the two Koszul complexes are isomorphic. \ Then $\bm{T}$ is Taylor invertible, as desired.
\end{proof}

\begin{corollary} \label{Cor2}
Let $\bm{T}$, $0 < t \le 1$ and $\Delta_t(\bm{T})$ be as above. \ The following statements are equivalent. \newline
(i) $\bm{T}$ is Harte invertible. \newline
(ii) $\Delta_t(\bm{T})$ is Harte invertible.
\end{corollary}

\begin{proof}
Assume first that $\bm{T}$ is Harte invertible. \ Then it is left invertible, and by Theorem \ref{thmP}, the operator $P$ is invertible. \ As a result, the Koszul complexes $K(\bm{T},\mathcal{H})$ and $K(\Delta_t(\bm{T}),\mathcal{H})$ are isomorphic. \ It follows that $\Delta_t(\bm{T})$ is Harte invertible.

Conversely, assume that $\Delta_t(\bm{T})$ is Harte invertible, and therefore right invertible. \ By Theorem \ref{thmrP}, $P$ is invertible, and it follows that the two Koszul complexes are isomorphic. \ Then $\bm{T}$ is Harte invertible, as desired.
\end{proof}

The following corollary is a simple consequence of the preceding results. \ We omit the proof.

\begin{corollary} \label{Cor3}
Let $\bm{T}$, $0 < t \le 1$ and $\Delta_t(\bm{T})$ be as above, and fix $0 \le k \le d$. \ The following statements hold. \newline
(i) \ $\bm{0} \notin \sigma_{\pi,k}(\Delta_t(\bm{T})) \Leftrightarrow \bm{0} \notin \sigma_{\pi,k}(\bm{T})$. \newline
(ii) \ If, in addition, $t<1$ then $\bm{0} \notin \sigma_{\delta,k}(\Delta_t(\bm{T})) \Leftrightarrow \bm{0} \notin \sigma_{\delta,k}(\bm{T})$.
\end{corollary}

\subsection{Fredholmness when $\bm{\lambda}=\bm{0}$} \label{Fredholm}
\ Continuing with the case $\bm{\lambda}=\bm{0}$, we will now study the Fredholmness of $\bm{T}$ and $\Delta_t(\bm{T})$. \ By now, the analogues of Theorems \ref{thmP} and \ref{thmrP} should be natural. \ We will give an abridged proof of Theorem \ref{EthmP}, and leave the proof of Theorem \ref{EthmrP} to the interested reader.

\begin{theorem} \label{EthmP}
Let $\bm{T}$ be a commuting $d$--tuple, and let $K(\bm{T},\mathcal{H})$ and $P$ be as above. \ Assume that $0<t<1$. \ The following statements are equivalent: \newline
(i) \ $\bm{T}$ is left Fredholm; \newline
(ii) \ $P$ is Fredholm; \newline
(iii) \ $\Delta_t(\bm{T})$ is left Fredholm.
\end{theorem}

\begin{proof}[Sketch of Proof]
From page \pageref{page5}, recall the co-chain homomorphism 
$$
\Phi_t:K(\bT,\mathcal{H}) \rightarrow K(\dbT,\mathcal{H})
$$
induced by $P^t$. \ Let $\pi$ denote the Calkin map, and let $\pi(P^t)$ denote the corresponding element in $\mathcal{Q}(\mathcal{H})$. \ At the Calkin algebra level, $\Phi_t$ becomes $\varphi_t$, a co-chain homomorphism (induced by $L_{\pi(P^t)}$) between the Koszul complexes $K(L_{\pi(\bm{T})},\QH)$ and $K(L_{\pi(\Delta_t(\bm{T}))},\QH)$, where $\pi(\bm{T})$ and $\pi(\Delta_t(\bm{T}))$ have the obvious definitions, and $L_{\bigcdot}$ denotes the $d$--tuple of left multiplications by the coordinates of $\bigcdot$ . \newline  
$(ii) \Rightarrow (iii)$: \ Assume that $P$ is Fredholm. \ Then $P^t$ is Fredholm, and therefore $\pi(P^t)$ is an invertible element of $\mathcal{Q}(\mathcal{H})$. \ It follows that the associated map $\varphi_t$ is an isomorphism. \ Since $\bT$ is left Fredholm, $L_{\pi(\bm{T})}$ is left invertible, and therefore $L_{\pi(\Delta_t(\bm{T}))}$ is left invertible, which implies that $\dbT$ is left Fredholm.
\end{proof}

\begin{theorem} \label{EthmrP}
Let $\bm{T}$ be a commuting $d$--tuple, and let $K(\bm{T},\mathcal{H})$ and $P$ be as above. \ Assume that $0<t\le 1$, and consider the following statements. \newline
(i) \ $\bm{T}$ is right Fredholm; \newline
(ii) \ $P$ is Fredholm; \newline
(iii) \ $\Delta_t(\bm{T})$ is right Fredholm. \newline
Then $(iii) \Rightarrow (ii)$, and $(iii) \Rightarrow (i)$. \ The implication $(i) \Rightarrow (iii)$ is not true in general. 
\end{theorem}

\begin{corollary} \label{Cor4}
Let $\bm{T}$, $0 < t \le 1$ and $\Delta_t(\bm{T})$ be as above. \ The following statements are equivalent. \newline
(i) $\bm{T}$ is Fredholm. \newline
(ii) $\Delta_t(\bm{T})$ is Fredholm.
\end{corollary}

\subsection{Taylor invertibility when $\bm{\lambda} \ne \bm{0}$} \label{Taylor2}
\ We now consider the various forms of invertibility of $\bT-\bm{\lambda}$ when $\bm{\lambda} \ne \bm{0}$. \ From page \pageref{crisscross}, recall that two (not necessarily commuting) $d$--tuples $\bm{A}$ and $\bm{B}$ criss--cross commute if $A_iB_jA_k=A_kB_jA_i$ and $B_iA_jB_k=B_kA_jB_i$ for all $i,j,k=1,\cdots,d$. \ Recall also that, under this condition, both $\bm{AB}$ and $\bm{BA}$ are commuting. 

\begin{lemma} \label{BenZe}
(cf. \cite{BenZe1,BenZe2}) \ Assume that $\bm{A}$ and $\bm{B}$ criss--cross commute. \ Then 
$$
\sigma_T(\bm{BA}) \setminus \{\bm{0}\} =\sigma_T(\bm{AB}) \setminus \{\bm{0}\}.
$$
\end{lemma}

\subsection{Fredholmness when $\bm{\lambda} \ne \bm{0}$} \label{Fredholm2}
Here we need the analogue of Lemma \ref{BenZe} for the Taylor essential spectrum.

\begin{lemma} \label{Li100}
(cf. \cite{BenZe1,BenZe2,Li1,Li2}) Assume that $\bm{A}$ and $\bm{B}$ criss--cross commute. \ Then 
$$
\sigma_{Te}(\bm{BA}) \setminus \{\bm{0}\} =\sigma_{Te}(\bm{AB}) \setminus \{\bm{0}\}.
$$
Moreover, for $\bm{0} \ne \bm{\lambda} \notin \sigma_{Te}(\bT)$, one has 
$$
\index(\dbT-\bm{\lambda})=\index(\bT-\bm{\lambda}).
$$
\end{lemma}

\subsection{The Spectral Picture of $\bT$ and $\dbT$} \ We are now ready to prove Theorems \ref{Taylorsp}, \ref{Tayloressentialsp} and \ref{othersp}, which we will restate as a single result, encompassing the various spectral systems.

\begin{theorem} \label{Taylorsp2}
Let $\bm{T}=\bm{(}T_{1},\cdots,T_{d})$ be a commuting $d$--tuple of operators on $\mathcal{H}$, and let $0 < t \leq 1$. \ Then
\begin{itemize}
\item[(\underline{Taylor})]
\begin{equation*}
\sigma _{T}(\Delta _{t}\left( \bm{T}\right))=\sigma _{T}(\bm{T}).\medskip
\end{equation*}
As a consequence, we have 
$$
r(\Delta _{t}\left( \bm{T}\right))=r_{T}\left( \bm{T}\right),
$$
where $r_T$ denotes the joint spectral radius.

\begin{changemargin}{-0.94cm}{0cm}
\item[(\underline{Taylor ess.})]
\begin{equation*}
\sigma _{Te}(\Delta _{t}\left( \bm{T}\right) )=\sigma _{Te}(\bm{T}).\medskip
\end{equation*}
\item[($\underline{\index}$)] \ The Fredholm index satisfies
$$
\index (\Delta _{t}(\bm{T})-\bm{\lambda})=index (\bm{T}-\bm{\lambda}),
$$
for all $\bm{\lambda} \notin \sigma _{Te}(\bm{T})$. \medskip

\item[(\underline{Other $\sigma$'s}): (i)] \ $\sigma _{p}(\Delta _{t}\left( \bm{T}\right) )=\sigma _{p}(\bm{T})$.
\item[(ii)] \ $\sigma _{\ell}(\Delta _{t}\left( \bm{T}\right) )=\sigma _{\ell}(\bm{T})$.
\item[(iii)] \ $\sigma _{H}(\Delta _{t}\left( \bm{T}\right) )=\sigma _{H}(\bm{T})$.
\item[(iv)] \ Assume $t>0$. \ Then $\sigma _{r}(\bm{T}) \subseteq \sigma _{r}(\Delta _{t}( \bm{T})) \subseteq \sigma_r(\bm{T}) \cup \{0\}$.
\item[(v)] \ More generally, if 
$\sigma_{\pi,k}$ and $\sigma_{\delta,k}$ denote the S\l odkowski's spectral systems, we have
$$
\sigma _{\pi,k}(\Delta _{t}\left( \bm{T}\right) )=\sigma _{\pi,k}(\bm{T}) \; (k=0,\dots,d)
$$
and 
$$
\sigma_{\delta,k}(\bm{T}) \subseteq \sigma _{\delta,k}(\Delta _{t}\left( \bm{T}\right) )\subseteq \sigma _{\delta,k}(\bm{T}) \cup \{0\} \; (k=0,\cdots,d).
$$
\end{changemargin}

\end{itemize}
\end{theorem}

\begin{proof}
Let $\bm{\lambda} \in \sigma _{T}(\Delta _{t}\left( \bm{T}\right))$. \ If $\bm{\lambda} = \bm{0}$, we know that $\dbT$ is not Taylor invertible. \ By Corollary \ref{Cor1}, $\bT$ is not Taylor invertible, and therefore $\bm{\lambda} \in \sigma_T(\bT)$. \ Similarly, $\bm{0} \in \sigma_T(\bT) \Rightarrow \bm{0} \in \sigma_T(\dbT)$.

Assume now that $\bm{\lambda} \ne \bm{0}$, and recall that $\bT=\bm{AB}$ and $\dbT=\bm{BA}$ for $\bm{A}:=(V_1P^{1-t},\cdots,V_dP^{1-t})$ and $\bm{B}:=(P^t,\cdots,P^t)$ (cf. page \pageref{ABBA}). \ By Lemma \ref{BenZe}, the Taylor invertibility of $\dbT-\bm{\lambda}$ is equivalent to the Taylor invertibility of $\bT-\bm{\lambda}$. \ This completes the proof of the equality of $\sigma_T(\dbT)$ and $\sigma_T(\bT)$.

To deal with the Taylor essential spectra, we can adapt the above arguments, replacing Taylor invertibility by Fredholmness, Corollary \ref{Cor1} by Corollary \ref{Cor4}, and Lemma \ref{BenZe} by Lemma \ref{Li100}, respectively. \ In terms of the Fredholm index, the results in \cite{Li1} are based on an isomorphism between the homology spaces of $\KdT$ and $\KT$, which guarantees the equality of the Fredholm indices away from $\bm{0}$. \ However, the Fredholm index is continuous on the complement in $\mathbb{C}^d$ of $\sigma_T(\bT)$ \cite{Cu2,Appl}, and integer-valued, and therefore $\index (\Delta _{t}(\bm{T})-\bm{\lambda})=\index (\bm{T}-\bm{\lambda})$ must also hold whenever $\bm{0} \notin \sigma_T(\bT)$.

As for other spectral systems, we can use Remarks \ref{remP} and \ref{remEP} together with the fact that the results in \cite{BenZe1,BenZe2,Li1,Li2} are based on isomorphisms of the Koszul complexes, or of the related homology spaces. \ However, for the case of the point spectrum $\sigma_p$ and $\bm{\lambda}=\bm{0}$, we can simply invoke the condition $\ker T_1 \cap \cdots \cap \ker T_d = \ker P = \ker V_1 \cap \cdots \cap \ker V_d$. \ For instance, if $0 \ne x \in \mathcal{H}$ and $T_1x= \cdots =T_dx=0$, then $Px=0$ and $V_1x= \cdots =V_dx=0$, which readily implies that $P^tV_1P^{1-t}x=\cdots =P^tV_dP^{1-t}x=0$. \ As a result, $\sigma_p(\bT) \subseteq \sigma_p(\dbT)$. \ Conversely, if $P^tV_1P^{1-t}x=\cdots =P^tV_dP^{1-t}x=0$ for some $0 \ne x \in \mathcal{H}$, then $P$ must not be one-to-one (otherwise we get a contradiction) and therefore $T_1x= \cdots =T_dx=0$. 
\end{proof}

We conclude this Section with a simple application of Theorem \ref{Taylorsp2}.

\begin{corollary}
Let $\bm{T}=\bm{(}T_{1},\cdots,T_{d})$ be a commuting $d$--tuple of operators on $\mathcal{H}$, and let $0 < t < 1$. \ Assume that $\bm{0} \in \sigma_r(\dbT) \setminus \sigma_r(\bT)$. \ Then $\bm{0} \in \sigma_{\ell}(\bT) \cap \sigma_{\ell}(\dbT)$.
\end{corollary}

\begin{proof}
First recall that the two Koszul complexes for $\bT$ and $\dbT$ are isomorphic when $P$ is invertible. \ By assumption, the Koszul complexes are not isomorphic, so $P$ is not invertible. \ Let $\{x_n\}_{n=1}^{\infty}$ be a sequence of unit vectors in $\mathcal{H}$ such that $Px_n \rightarrow 0$ as $n \rightarrow \infty$. \ Since $t<1$, it follows that $P^tV_iP^{1-t}x_n \rightarrow 0$ as $n \rightarrow \infty$, for all $i=1,\cdots,d$. \ This means that $\dbT$ is not bounded below; that is, $\bm{0} \in \sigma_{\ell}(\dbT)$. \ But $\bm{0}$ is also in $\sigma_{\ell}(\bT)$, because $P$ is not bounded below. \ It follows that $\bm{0} \in \sigma_{\ell}(\bT) \cap \sigma_{\ell}(\dbT)$, as desired.
\end{proof}


\section{Joint Spectral Radius of the Generalized Spherical Aluthge Transform} \label{JointSpecRad}

In this Section we will extend K. Feki and T. Yamazaki's proof of the spectral radius formula for the spherical Aluthge transform in $d$--variables ($d>1$ and $t=\frac{1}{2}$) \cite{FeYa} to the case of the generalized spherical Aluthge transform ($d>1$ and $0<t<1$). \ First, we present a simple example that shows that the spectral radius formula (\ref{eq17}) cannot be extended to the case of $t=1$, even for $d=1$.

\begin{example} \label{ex67}
Consider an operator $T \in \mathcal{B}(\mathcal{H})$ with polar decomposition $T \equiv VP$, where the partial isometry $V$ satisfies the algebraic equation $V^2=-I$. \ It is clear that $V$ must be unitary, and $V^*=-V$. Recall that, in general, the polar decomposition of the adjoint of $T$ is given by $T^* \equiv V^*Q$, where $Q:=\sqrt{TT^*}$. \ Thus, in the case at hand, we must have $Q=-VPV$ and $T^*=-PV$. \ It follows that 
$$
\Delta_1(T)=PV=-T^*,
$$
and therefore 
$$
\Delta_1(\Delta_1(T))=\Delta_1(-T^*)=\Delta_1(-V^*Q)=-QV^*=-(-VPV)(-V)=-VPV^2=VP=T.
$$
We conclude that
\begin{equation*}
\Delta _{1}^{(n)}(T)=\left\{
\begin{tabular}{ll}
$T$, & $n$ is even \\
$-T^*$, & $n$ is odd.
\end{tabular}%
\right.
\end{equation*}
As a consequence, $\left\|\Delta_1^{(n)}(T)\right\|=\left\|T\right\|$ for all $n \ge 1$, which implies that $\lim_{n \rightarrow\infty}\left\|\Delta_1^{(n)}(T)\right\|=\left\|T\right\|$. \ On the other hand, within this collection of operators, $r(T)$ may be smaller than $\left\|T\right\|$, and therefore, $r(T) \ne \lim_{n \rightarrow\infty}\left\|\Delta_1^{(n)}(T)\right\|$. \ For a concrete example where $r(T) < \left\|T\right\|$, let $\mathcal{H}:=\mathbb{C}^2$, let $1< k \in \mathbb{R}$ and let
\begin{equation*}
T_k :=\left(
\begin{array}{cc}
1 & k \\
-k & -1%
\end{array}%
\right) ,
\end{equation*}
with polar decomposition
\begin{equation*}
T_k \equiv VP =\left(
\begin{array}{cc}
0 & 1 \\
-1 & 0
\end{array}
\right) \left(
\begin{array}{cc}
k & 1 \\
1 & k
\end{array}
\right) . \qed
\end{equation*}
\end{example}

As we mentioned in the Introduction, T. Yamazaki proved, for single operators, that $r(T)=\lim_n\left\|\Delta_{1/2}^{(n)}(T)\right\|$. \ Subsequently, T.Y. Tam \cite{Tam}, building on D. Wang's \cite{Wang} simplified proof of Yamazaki's result, extended the spectral radius formula to the case of $0 <t<1$, under the assumption that $T$ is invertible. \ Very recently, K. Feki and T. Yamazaki have extended these results to the case of $d>1$ and $t=\frac{1}{2}$ \cite{FeYa}. \ In what follows, we will extend the spectral radius formula to all commuting $d$--tuples and all $0<t<1$.   

First, we need to record three operator inequalities that can be traced back to the work of E. Heinz \cite{Hei} and A. McIntosh \cite{McI}; for two of them, the formulation below is taken from \cite{Kit}.

\begin{lemma}
\label{lem5} \cite{McI} \ Let $\mathcal{H}$ be a complex Hilbert space and let
$A,B,X\in \mathcal{B}(\mathcal{H})$. \ Then,
\begin{equation}
\left\Vert A^{\ast }XB\right\Vert \leq \left\Vert AA^{\ast }X\right\Vert ^{%
\frac{1}{2}}\left\Vert XBB^{\ast }\right\Vert ^{\frac{1}{2}}.  \label{Ineq1}
\end{equation}%
\medskip
\end{lemma}

\begin{lemma}
\label{lem6} (\cite{Hei},\cite[Theorems 1 and 2]{Kit}) \ Let $\mathcal{H}$ be a complex Hilbert space and
let $A,B,X\in \mathcal{B}(\mathcal{H})$. \ Then, for $t\in \left[ 0,1\right],
$
\begin{equation}
\left\Vert A^{t}XB^{t}\right\Vert \leq \left\Vert AXB\right\Vert
^{t}\left\Vert X\right\Vert ^{1-t}\medskip  \label{Ineq2}
\end{equation}%
and%
\begin{equation}
\left\Vert A^{t}XB^{1-t}\right\Vert \leq \left\Vert AX\right\Vert
^{t}\left\Vert XB\right\Vert ^{1-t}\medskip  \label{Ineq3}
\end{equation}
\end{lemma}

We will also need the multivariable analogue of the classical spectral radius formula, established by V. M\"uller and A. Soltysiak in 1992. \ For this, first recall that the joint spectral radius of a commuting $d$--tuple is given by $r(\bm{T}):=\sup \{\left\|\bm{\lambda}\right\|_2: \bm{\lambda} \in \sigma_T(\bT) \}$ (cf. \cite{Bun,MuSo,ChZe}). \ It is straightforward to see that $r(\bT) \le \left\|\bT\right\|$. \ Also, given a $d$--tuple $\bm{T}$ and an integer $k \ge 1$, we will let $\bm{T}^k$ denote the $d^k$--tuple defined inductively as follows:
\begin{eqnarray} \label{defpower}
\bm{T}^1&:=&\bm{T} \nonumber \\
\bm{T}^{2}&:=&(T_1T_1,T_1T_2,\cdots,T_1T_d,T_2T_1,T_2T_2,\cdots,T_dT_1,T_dT_2,\cdots,T_dT_d), \\
&\vdots& \nonumber \\
\bm{T}^{k+1}&:=&\bm{T}\bm{T}^k. \nonumber
\end{eqnarray}

\begin{lemma} \label{lemMuSo} \cite{MuSo} \ Let $\bT$ be a commuting $d$--tuple of Hilbert space operators. \ Then
$$
r(\bT)=\lim_{k \rightarrow \infty}\left\|\bT^k\right\|_2^{\frac{1}{k}} .
$$
\end{lemma}

As a way to initiate our discussion, recall that, for $T \in \mathcal{B}(\mathcal{H})$ and $0<t\le1$, one has $\left\Vert \Delta _{t}\left( T\right) \right\Vert \leq \left\Vert T\right\Vert $ \cite{DKY,Tam}. \ (Briefly, $\left\|P^tVP^{1-t}\right\|
\le \left\|PV\right\|^t\left\|VP\right\|^{1-t}$ (by (\ref{Ineq3})), and it is always true that $\left\|PV\right\|=\left\|V^*P\right\|\le\left\|P\right\|=\left\|T\right\|=\left\|VP\right\|$.) \ As a consequence, the sequence of iterates of $\Delta_t(T)$ has decreasing norms, and therefore $\lim\left\|\Delta_t^{(n)}(T)\right\|_2$ exists and is majorized by $\left\|T\right\|_2$.

The other ingredient that enters into the $1$--variable proof is the comparison between the norms of the integer powers of two consecutive iterates, $\left\|(\Delta_t^{(n+1)}(T))^{k}\right\|_2$ and  $\left\|(\Delta_t^{(n)}(T))^{k}\right\|_2$. \ Our goal is to provide the appropriate analogues in the case of $d$--variables.

\medskip
For the reader's convenience, we now restate Theorem \ref{thm18}.

\begin{theorem} \label{thm18A}
Let $\bm{T}=\bm{(}T_{1},\cdots,T_{d})$ be a commuting $d$--tuple of operators on $\mathcal{H}$. \ For $0 < t < 1$, we have%
\begin{equation} \label{eq17A}
r_{T}\left( \bm{T}\right) = \lim_{n\rightarrow \infty } \left\Vert \Delta _{t}^{\left( n\right) }\left( \bm{T}\right) \right\Vert
_{2} .\medskip
\end{equation}
\end{theorem}

To formulate a strategy to prove this result, we remind ourselves that there are three elements that should enter into the proof: (i) an application of M\"uller and Soltysiak's joint spectral radius formula \cite{MuSo}, which will require a good handle on the powers of $\bT$, its generalized spherical Aluthge transform $\dbT$ and its iterates; (ii) the equality of the Taylor spectra of $\bT$ and any iterate $\Delta_T^{(n)}(\bT)$; and (iii) the asymptotic behavior of the norms of the iterates. \ The first goal is to establish the inequality $r(\bT) \le \left\|\Delta_t^{(n)}(\bT)\right\|_2 $ for all $t \in \left(0,1\right), n \ge 1, k \ge 1.$ \ Once this is done, we will look carefully at the behavior of $\left\|(\Delta_t^{(n)}(\bT))^k\right\|_2$ as a function of the parameters $t$, $n$ and $k$. \ Our proof will be structured around nine key steps, each of independent interest; we will record them in a series of nine auxiliary lemmas. 

\begin{lemma} \label{Step1}
(cf. \cite[Lemma 4.5]{FeYa}) \ For $k \ge 1$, 
$$
\left\|\bT^{k+1}\right\|_2 \le \left\|\bT\right\|_2 \left\|\bT^{k}\right\|_2 .
$$
\end{lemma}

\begin{proof}
By (\ref{defpower}), we have 
\begin{eqnarray*}
\left\|\bT^{k+1}\right\|_2^2&=&\left\|\sum\limits_{i_1,\cdots,i_{k+1}=1}^d T_{i_1}^*\cdots T_{i_{k+1}}^*T_{i_{k+1}}\cdots T_{i_1}\right\| \\
&=&\left\|\sum\limits_{i_1=1}^d T_{i_1}^* \left(\sum\limits_{i_2,\cdots,i_{k+1}=1}^dT_{i_2}^*\cdots T_{i_{k+1}}^*T_{i_{k+1}}\cdots T_{i_2} \right) T_{i_1}\right\| \\
&\le&\left\|\sum\limits_{i_1=1}^d T_{i_1}^*T_{i_1}\right\|\left\|\sum\limits_{i_2,\cdots,i_{k+1}=1}^dT_{i_2}^*\cdots T_{i_{k+1}}^*T_{i_{k+1}}\cdots T_{i_2} \right\| \\
&=&\left\|\bT\right\|_2^2\left\|\bT^k\right\|_2^2,
\end{eqnarray*}
where the inequality in the third line requires an application of the well-known inequality $\left\|X^*AX\right\|\le\left\|X^*X\right\|\left\|A\right\|$, for $A \ge 0$ and $X$ arbitrary; in the present situation, $X$ is the column vector with operator entries which are products of $T_1,\cdots,T_d$, and $A$ is the diagonal operator matrix with constant diagonal entry $\sum\limits_{i_2,\cdots,i_{k+1}=1}^dT_{i_2}^*\cdots T_{i_{k+1}}^*T_{i_{k+1}}\cdots T_{i_2}$.
\end{proof}
  
\begin{lemma} \label{Step2}
For $0 < t \le 1$, 
$$
\left\|\dbT\right\|_2 \le \left\|\bT\right\|_2 .
$$
\end{lemma}

\begin{proof} 
On the Hilbert space $\mathcal{H}^d:=\mathcal{H} \oplus \cdots \oplus \mathcal{H}$ ($d$ orthogonal summands), consider the operators $A:=\diag(P,\cdots,P)$, $B:=A$, and
\begin{equation*}
X:=\left(
\begin{array}{cccc}
V_{1} & 0 & 0 & 0 \\
V_{2} & 0 & 0 & 0 \\
\vdots & \vdots & \ddots & \vdots \\
V_{d} & 0 & 0 & 0
\end{array}
\right) .
\end{equation*}
Clearly, $\dbT$ is the first column of the operator matrix $A^tXB^{1-t}$, whose other columns are all zero. \ Now apply (\ref{Ineq3}) to obtain
$$
\left\|\dbT\right\|_2 = \left\|A^tXB^{1-t}\right\|\le \left\|AX\right\|^t\left\|XB\right\|^{1-t}.
$$
Observe now that 
$$
\left\|AX\right\|=\left\|X^*A\right\| \le \left\|A\right\|=\left\|P\right\|=\left\|\bT\right\|_2,
$$
and $\left\|XB\right\|=\left\|\bT\right\|_2$. \ It follows that  
$$
\left\|\dbT\right\|_2 \le \left\|\bT\right\|_2^t\left\|\bT\right\|_2^{1-t} = \left\|\bT\right\|_2,
$$
as desired.
\end{proof}

\begin{lemma} (cf. \cite[Lemma 4.3]{FeYa} for $t=\frac{1}{2}$) \label{Step3}
\ For $0 < t \le 1$ and $k \ge 1$, 
$$
\left\|(\dbT)^k\right\|_2 \le \left\|\bT^k\right\|_2 .
$$
\end{lemma}

\begin{proof}
We imitate the Proof of Lemma \ref{Step2}, this time letting $A,B,X$ act on the space $\mathcal{H}^{d^k}$. \ Thus, $A:=\diag(P,\cdots,P)$, $B:=A$, and $X$ is a $d^k$ by $d^k$ operator matrix, with the only nonzero entries appearing in the first column; the entries in that column are given by $V_{i_1}P\cdots PV_{i_k}$, for $i_1,\cdots,i_k=1,\cdots,d$. \ It is easy to show that $\left\|(\dbT)^k\right\|_2=\left\|A^tXB^{1-t}\right\|$. \ We can once again appeal to (\ref{Ineq3}) to obtain
\begin{equation} \label{newineq}
\left\|(\dbT)^k\right\|_2 = \left\|A^tXB^{1-t}\right\|\le \left\|AX\right\|^t\left\|XB\right\|^{1-t}.
\end{equation}
(In the special case of $t=\frac{1}{2}$, one can use instead Lemma \ref{lem5}, as done in \cite{FeYa}.) \ Now, 
\begin{eqnarray}
\left\|AX\right\|^2&=&\left\|X^*AX\right\| \nonumber \\
&=&\left\|\sum\limits_{i_1,\cdots,i_k}^d V_{i_1}^*P\cdots PV_{i_k}^*P^2V_{i_k}P\cdots PV_{i_1}\right\| \nonumber \\
&=&\left\|\sum\limits_{i_1,\cdots,i_k}^d V_{i_1}^*P\cdots PV_{i_k}^*P\left(\sum\limits_{i_1,\cdots,i_{k+1}}^d V_{i_{k+1}}^dV_{i_{k+1}}\right) PV_{i_k}P\cdots PV_{i_1}\right\| \nonumber \\
&& (\textrm{since } V_1^*V_1+\cdots+V_d^*V_d \textrm{ is the projection onto the closure of } \Ran P) \nonumber \\
&=& \left\|\sum\limits_{i_1,\cdots,i_k}^d V_{i_1}^* \left(\sum\limits_{i_2,\cdots,i_{k+1}}^d T_{i_2}^* \cdots T_{i_{k+1}}^*T_{i_{k+1}} \cdots T_{i_2} \right) V_{i_1}\right\| \nonumber \\
& \le & \left\|\sum\limits_{i_2,\cdots,i_{k+1}}^d T_{i_2}^* \cdots T_{i_{k+1}}^*T_{i_{k+1}} \cdots T_{i_2}\right\| \nonumber \\
&=&\left\|\bT^k\right\|_2^2. \label{newIneq}
\end{eqnarray}
On the other hand, it is straightforward to verify that $\left\|XA\right\|=\left\|\bT^k\right\|$. \ Inserting this and (\ref{newIneq}) in (\ref{newineq}) yields the desired conclusion.
\end{proof}

\begin{corollary} \label{cork+1}(cf. \cite[bottom of page 13]{FeYa}) \ With the notation in Lemma \ref{Step3}, we also have:
$$
\left\|AXA\right\|=\left\|\bT^{k+1}\right\|_2.
$$
\end{corollary}

\begin{proof}
\begin{eqnarray*}
\left\|AXA\right\|^2&=&\left\|AX^*A^2XA\right\| \\
&=&\left\|\sum\limits_{i_1,\cdots,i_k}^d PV_{i_1}^*P\cdots PV_{i_k}^*P^2V_{i_k}P\cdots PV_{i_1}P\right\| \\
&=&\left\|\sum\limits_{i_1,\cdots,i_k}^d PV_{i_1}^*P\cdots PV_{i_k}^*P\left(\sum\limits_{i_1,\cdots,i_{k+1}}^d V_{i_{k+1}}^dV_{i_{k+1}}\right) PV_{i_k}P\cdots PV_{i_1}P\right\| \\
& \le & \left\|\sum\limits_{i_1,\cdots,i_{k+1}}^d T_{i_1}^* \cdots T_{i_{k+1}}^*T_{i_{k+1}} \cdots T_{i_1}\right\| \\
&=&\left\|\bT^{k+1}\right\|_2^2. \label{newIneq2}
\end{eqnarray*}
\end{proof}

\begin{lemma} \label{Step4}
For $0<t \le 1$ and $k,n \ge 1$, 
$$
\left\|(\Delta_t^{(n+1)}(\bT))^k\right\|_2 \le \left\|(\Delta_t^{(n)}(\bT))^k\right\|_2 .
$$
\end{lemma}

\begin{proof}
Since $\Delta_t^{(n+1)}(\bT)=\Delta_t(\Delta_t^{(n)}(\bT))$, it suffices to prove the conclusion for $n=0$. \ But this is the content of Lemma \ref{Step3}, when applied to $\Delta_t^{(n)}(\bT)$. 
\end{proof}

\begin{corollary} \label{cordecrease}
For $0<t \le 1$ and $k,n \ge 1$, 
$$
\left\|(\Delta_t^{(n)}(\bT))^k\right\|_2 \le \cdots \le \left\|(\Delta_t^{(2)}(\bT))^k\right\|_2 \le \left\|(\Delta_t(\bT))^k\right\|_2 \le \left\|\bT^k\right\|_2.
$$
\end{corollary}

\begin{proof}
Immediate from Lemmas \ref{Step3} and \ref{Step4}. 
\end{proof}

\begin{lemma} \label{Step5}
For $0 < t <1$ and $k,n \ge 1$, 
$$
r(\bT) \le \left\|(\Delta_t^{(n)}(\bT))^k\right\|_2^{\frac{1}{k}} .
$$
\end{lemma}

\begin{proof}
By Theorem \ref{Taylorsp}, we know that $\sigma_T(\dbT)=\sigma_T(\bT)$. \ It follows that the same is true for any iterate of $\dbT$; that is, $\sigma_T(\Delta_t^{(n)}(\bT))=\sigma_T(\bT)$ for all $n \ge 1$. \ Also, let us recall that the Spectral Mapping Theorem holds for $\sigma_T$ \cite{Appl,Tay2,Vas}. \ We claim that $r(\bT^k)=(r(\bT))^k$. \ For, using (\ref{defpower}) we have
\begin{eqnarray*}
\sigma_T(\bT^k)&=&\sigma_T(T_1^k,T_1^{k-1}T_2,\cdots,T_{d-1}T_d^{k-1}) \\
&=&\{(\lambda_1^k,\lambda_1^{k-1}\lambda_2,\cdots,\lambda_{d-1}\lambda_d^{k-1},\lambda_d^k):(\lambda_1,\cdots,\lambda_d) \in \sigma_T(\bT) \}.
\end{eqnarray*}
Since 
$$
\left\|(\lambda_1^k,\lambda_1^{k-1}\lambda_2,\cdots,\lambda_{d-1}\lambda_d^{k-1},\lambda_d^k)\right\|_2^2=(\left|\lambda_1\right|^2+\cdots \left|\lambda_d\right|^2)^k=\left\|\bm{\lambda}\right\|_2^{2k},
$$
it follows that $r(\bT^k)=(r(\bT))^k$. \ Therefore, 
\begin{eqnarray*}
(r(\bT))^k&=&(r(\Delta_t^{(n)}(\bT))^k=r((\Delta_t^{(n)}(\bT))^k) \\
&&(\textrm{by the above--mentioned Spectral Mapping Property applied to } \Delta_t^{(n)}(\bT)) \\
&\le&\left\|(\Delta_t^{(n)}(\bT))^k\right\|_2.
\end{eqnarray*}
It follows that $r(\bT) \le \left\|(\Delta_t^{(n)}(\bT))^k\right\|_2^{\frac{1}{k}}$, as desired.
\end{proof}

\begin{remark}
Corollary \ref{cordecrease} and Lemma \ref{Step5} provide a proof of Theorem \ref{thm17}. \qed
\end{remark}

\begin{lemma} \label{Step6}
For $0 < t <1$ and $k \ge 1$, $\left(\left\|(\dbT^{(n)})^k\right\|_2\right)_{n=1}^{\infty}$ is a decreasing sequence. \ If we let $L_{t,k}$ denote its limit, then $r(\bT)^k \le L_{t,k}$.
\end{lemma}

\begin{proof}
This is a straightforward consequence of Corollary \ref{cordecrease} and Lemma \ref{Step5}. 
\end{proof}

\begin{lemma} \label{Step7}
For $0 \le t \le \frac{1}{2}$ and $k \ge 1$, we have
$$
\left\|(\dbT)^k\right\|_2 \le \left\|\bT^{k+1}\right\|_2^t \left\|\bT^{k-1}\right\|_2^{1-t} \left\|\bT\right\|_2^{1-2t}.
$$ 
\end{lemma}

\begin{proof}
With the notation used in the Proof of Lemma \ref{Step3}, we have
\begin{eqnarray*}
\left\|(\dbT)^k\right\|_2&=&\left\|A^tXA^{1-t}\right\|=\left\|(A^tXA^t)A^{1-2t}\right\| \\
                         &\le&\left\|A^tXA^t\right\|\left\|A^{1-2t}\right\| \\
												 & \le &	\left\|AXA\right\|^t\left\|X\right\|^{1-t} \left\|\bT^{1-2t}\right\|_2 \\
												&	\le & \left\|\bT^{k+1}\right\|_2\left\|\bT^{k-11}\right\|_2\left\|\bT^{1-2t}\right\|_2,
\end{eqnarray*}																	
using Corollary \ref{cork+1} to get the first factor in the last line.
\end{proof}

\begin{lemma} \label{Step8}
For $\frac{1}{2} \le t \le 1$ and $k \ge 1$, we have
$$
\left\|(\dbT)^k\right\|_2 \le \left\|\bT^{k+1}\right\|_2^{1-t} \left\|\bT^{k-1}\right\|_2^{t} \left\|\bT\right\|_2^{2t-1}.
$$ 
\end{lemma}

\begin{proof}
As in Lemma \ref{Step7},
\begin{eqnarray*}
\left\|(\dbT)^k\right\|_2&=&\left\|A^tXA^{1-t}\right\|=\left\|A^{2t-1}(A^{1-t}XA^{1-t})\right\| \\
                         &\le&\left\|A^{2t-1}\right\|\left\|A^{1-t}XA^{1-t}\right\| \\
												 & \le &\left\|A^{2t-1}\right\|	\left\|AXA\right\|^{1-t} \left\|X\right\|^{t} \\
												&	\le & \left\|\bT\right\|_2^{2t-1}\left\|\bT^{k+1}\right\|_2\left\|\bT^{k-1}\right\|_2,
\end{eqnarray*}
as desired.
\end{proof}

\begin{lemma} \label{Step9}
Let $\bT$ be a commuting $d$--tuple of operators on Hilbert space, and assume that $r(\bT) > 0$. \ Then, for $0 < t <1$ and $k \ge 1$, we have 
$$
L_{t,k}=L_{t,1}^k .
$$
\end{lemma}

\begin{proof}
We will split the proof into two cases, and use induction on $k$.

{\bf Case 1}: $0 < t \le \frac{1}{2}$. \ For $k=1$, we use Lemma \ref{Step6}. \ Assume now that $L_{t,k}=L_{t,1}^k$ for all integers $k \le m$; we will establish the identity for $k=m+1$. \ By Lemma \ref{Step7}, with $\bT$ replaced by $\Delta_t^{(n)}(\bT)$, we have
\begin{eqnarray}
\left\|(\Delta_t^{(n+1)}(\bT))^m\right\|_2 &\le& \left\|(\Delta_t^{(n)}(\bT))^{m+1}\right\|_2^t \left\|(\Delta_t^{(n)}(\bT))^{m-1}\right\|_2^{1-t} \left\|\Delta_t^{(n)}(\bT)\right\|_2^{1-2t} \label{id1} \\
&\le&\left\|(\Delta_t^{(n)}(\bT))^{m}\right\|_2^t \left\|(\Delta_t^{(n)}(\bT))\right\|_2^t\left\|(\Delta_t^{(n)}(\bT))^{m-1}\right\|_2^{1-t} \left\|\Delta_t^{(n)}(\bT)\right\|_2^{1-2t} \nonumber \\
&& (\textrm{ by Lemma \ref{Step1}}) \nonumber \\
&=&\left\|(\Delta_t^{(n)}(\bT))^{m}\right\|_2^t \left\|(\Delta_t^{(n)}(\bT))^{m-1}\right\|_2^{1-t} \left\|\Delta_t^{(n)}(\bT)\right\|_2^{1-t} \label{id2}.
\end{eqnarray}
We now take the limit as $n \rightarrow \infty$, in (\ref{id1}) and (\ref{id2}) to obtain
$$
L_{t,m} \le L_{t,m+1}^t L_{t,m-1}^{1-t} L_{t,1}^{1-2t} \le L_{t,m}^t L_{t,m-1}^{1-t} L_{t,1}^{1-t}.
$$
In view of the inductive hypothesis, and letting $L:=L_{t,1}$, we can write
$$
L^m \le L_{t,m+1}^t L^{(m-1)(1-t)}L^{1-2t} \le L^{mt} L^{(m-1)(1-t)} L^{1-t};
$$
equivalently,
$$
L^{m} \le L_{t,m+1}^t L^{m-t(m+1)} \le L^m.
$$
It follows that $L_{t,m+1}^t L^{m-t(m+1)} = L^m$. \ By Lemma \ref{Step6}, $L \ge r(\bT) > 0$, so we obtain $L_{t,m+1}=L^{m+1}$, as desired.

{\bf Case 2}: $\frac{1}{2} \le t <1$. \ Here,
\begin{eqnarray}
\left\|(\Delta_t^{(n+1)}(\bT))^m\right\|_2 &\le& \left\|(\Delta_t^{(n)}(\bT))^{m+1}\right\|_2^{1-t} \left\|(\Delta_t^{(n)}(\bT))^{m-1}\right\|_2^t \left\|\Delta_t^{(n)}(\bT)\right\|_2^{2t-1} \label{id3} \\
&\le&\left\|(\Delta_t^{(n)}(\bT))^{m}\right\|_2^{1-t} \left\|(\Delta_t^{(n)}(\bT))\right\|_2^{1-t}\left\|(\Delta_t^{(n)}(\bT))^{m-1}\right\|_2^{t} \left\|\Delta_t^{(n)}(\bT)\right\|_2^{2t-1} \nonumber \\
&& (\textrm{ by Lemma \ref{Step1}}) \nonumber \\
&=&\left\|(\Delta_t^{(n)}(\bT))^{m}\right\|_2^{1-t} \left\|(\Delta_t^{(n)}(\bT))^{m-1}\right\|_2^{t} \left\|\Delta_t^{(n)}(\bT)\right\|_2^{t} \label{id4}.
\end{eqnarray}
As above, we then take the limit as $n \rightarrow \infty$, in (\ref{id3}) and (\ref{id4}) to obtain
$$
L_{t,m} \le L_{t,m+1}^{1-t} L_{t,m-1}^{t} L_{t,1}^{2t-1} \le L_{t,m}^{1-t} L_{t,m-1}^{t} L_{t,1}^{t}.
$$
In view of the inductive hypothesis, we then have 
$$
L^m \le L_{t,m+1}^{1-t} L^{(m-1)t}L^{2t-1} \le L^{m(1-t)} L^{(m-1)t} L^{t};
$$
equivalently,
$$
L^{m} \le L_{t,m+1}^{1-t} L^{(m+1)t-1} \le L^m.
$$
It follows that $L_{t,m+1}^t L^{m+(m+1)(t-1)} = L^m$. \ By Lemma \ref{Step6}, $L \ge r(\bT) > 0$, so we obtain $L_{t,m+1}=L^{m+1}$, as desired.
\end{proof}

We are now ready to put all the pieces together.

\begin{proof}[Proof of Theorem \ref{thm18A}] \ We first consider the case when $r(\bT) > 0$. \ By Corollary \ref{cordecrease} and Lemma \ref{Step9}, for each $k \ge 1$ the sequence $\{\left\|(\Delta_t^{(n)}(\bT))^k\right\|^{\frac{1}{k}}\}_{n=1}^{\infty}$ is decreasing to $L:=L_{t,1}$. \ Moreover, $L \ge r(\bT)$. \ We claim that $L=r(\bT)$. \ Suppose that $L > r(\bT)$. \ Fix $n_0 \ge 1$ and consider the sequence $\{\left\|(\Delta_t^{(n_0)}(\bT))^k\right\|_2^{\frac{1}{k}} \}_{k=1}^{\infty}$. \ By Lemmas \ref{Step6} and \ref{Step9}, this sequence is bounded from below by $L$. \ It follows that $\lim \inf_{k \rightarrow \infty} \left\|(\Delta_t^{(n_0)}(\bT))^k\right\|_2^{\frac{1}{k}} \ge L$. \ But we know that this limit is the spectral radius $r(\Delta_t^{(n)}(\bT))$, so $r(\bT) \ge L > r(\bT)$, a contradiction. 

Next, we look at the case when $r(\bT)=0$. \ On the orthogonal direct sum $\mathcal{H} \oplus \mathcal{H}$, let $\bm{S}:=\bT \oplus (cI,0,\cdots,0)$, where $c>0$. \ Clearly, $\sigma_T(\bm{S}) = \sigma_T(\bT) \cup \{(c,0,\cdots,0)\}=\{(c,0,\cdots,0)\}$, so $r(\bm{S})=c>0$. \ Moreover, $\Delta_t^{(n)}(\bm{S})=\Delta_t^{(n)}(\bm{T}) \oplus (cI,0,\cdots,0)$. \ It follows that $\left\|\Delta_t^{(n)}(\bT)\right\|_2 \le \left\|\Delta_t^{(n)}(\bm{S})\right\|_2$, and therefore
$$
\lim \sup \left\|\Delta_t^{(n)}(\bT)\right\|_2 \le \lim \sup \left\|\Delta_t^{(n)}(\bm{S})\right\|_2=r(\bm{S})=c.
$$
Since $c$ can be taken arbitrarily small, we conclude that $\lim \left\|\Delta_t^{(n)}(\bT)\right\|_2=0=r(\bT)$, as desired.
\end{proof}

\begin{remark}
Theorem \ref{thm16} is the $d=1$ instance of Theorem \ref{thm18A}. \ Similarly, Theorem \ref{thm15} follows from Lemma \ref{Step6} when $d=k=1$. \qed
\end{remark}

\end{document}